\documentclass[11pt,a4paper, twoside]{amsart}
\usepackage{amssymb}
\usepackage{graphicx}
\renewcommand{\epsilon}{\varepsilon}
\renewcommand{\setminus}{\smallsetminus}
\renewcommand{\emptyset}{\varnothing}

\newtheorem*{theorem*}{Theorem}
\newtheorem*{namedtheorem}{\theoremname}
\newcommand{\theoremname}{testing}

\newtheorem{theorem}{Theorem}[section]
\newtheorem{proposition}[theorem]{Proposition}

\newtheorem{lemma}[theorem]{Lemma}

\newtheorem{question}[theorem]{Question}

\newtheorem{thm}[theorem]{Theorem}
\newtheorem{cor}[theorem]{Corollary}

\newtheorem{remark}[theorem]{Remark}
\newtheorem{lem}[theorem]{Lemma}

\numberwithin{equation}{section}

\theoremstyle{definition}

\newcommand{\Z}{\mathbb Z}

\newcommand{\E}{\mathbb E}

\newcommand{\GL}{\operatorname{GL}}

\newcommand{\cohom}[3]{H^{{\raise1pt\hbox{$\scriptstyle#1$}}}(#2\>\!,#3)}
\newcommand{\tatecohom}[3]%
  {\widehat H^{{\raise1pt\hbox{$\scriptstyle#1$}}}(#2\>\!,#3)}

\newcommand{\Cohom}[3]%
  {H^{{\raise1pt\hbox{$\scriptstyle#1$}}}\big(#2\>\!,#3\big)}
\newcommand{\Tatecohom}[3]%
  {\widehat H^{{\raise1pt\hbox{$\scriptstyle#1$}}}\big(#2\>\!,#3\big)}

\newcommand{\homol}[3]{H_{{\lower1pt\hbox{$\scriptstyle#1$}}}(#2\>\!,#3)}
\newcommand{\homolog}[2]{H_{{\lower1pt\hbox{$\scriptstyle#1$}}}(#2)}

\DeclareMathOperator{\Out}{Out}

\DeclareMathOperator{\Aut}{Aut}

\newcommand{\lk}{\operatorname{lk}}
\newcommand{\st}{\operatorname{st}}

\numberwithin{equation}{section}

\newcommand{\q}{\mathbf q}

\newcommand{\D}{\mathcal{D}}

\def\P{\mathbf{P}}
\def\E{\mathbf{E}}

\def\unogrande{\text{\large\bf 1}}

\def\T{\mathcal T}

\newtheorem{theoremletra}{Theorem}

\begin{document}

\title[Trees, homology, and automorphism groups of RAAGs]{Trees, homology, and automorphism groups of RAAGs}

\author[J. Aramayona, J.\,L. Fern\'andez, P. Fern\'andez and C. Mart\'inez-P\'erez]{Javier Aramayona, Jos\'e L. Fern\'andez, Pablo Fern\'andez and Conchita Mart\'inez-P\'erez}

\address{Javier Aramayona: Department of Mathematics, Universidad Aut\'onoma de Madrid, 28049 Madrid, Spain.}
\email{javier.aramayona@uam.es}

\address{Jos\'e L. Fern\'andez: Department of Mathematics, Universidad Aut\'onoma de Madrid, 28049 Madrid, Spain.}
\email{joseluis.fernandez@uam.es}

\address{Pablo Fern\'andez: Department of Mathematics, Universidad Aut\'onoma de Madrid, 28049 Madrid, Spain.}
\email{pablo.fernandez@uam.es}

\address{Conchita Mart\'{\i}nez-P\'erez: Department of Mathematics,  Universidad de Zaragoza,  Pedro Cerbuna s/n,  50009 Zaragoza, Spain.}
\email{conmar@unizar.es}


\begin{abstract}
We study the homology of an explicit finite-index subgroup of the automorphism group of a partially commutative group, in the case when its defining graph is a tree. More concretely, we give a lower bound on the first Betti number of this subgroup, based on the number and degree of a certain type of vertices, which we call {\em deep}. We then use combinatorial methods to analyze the average value of this Betti number, in terms of the size of the defining tree.

\end{abstract}

\maketitle

\section{Introduction}

 Let $K$ be an (undirected) finite graph, and write $V(K)$ and $E(K)$ for its set of nodes and edges, respectively.
 The {\em righ-angled Artin group} (RAAG, for short) defined by $K$ is the group $A_K$ given by the presentation
 $$
 A_K=\langle a\in V(K) \mid [a,b]=1 \iff ab \in E(K) \rangle,
 $$
where $ab$ denotes the edge joining $a$ and $b$, and $[a,b]=aba^{-1}b^{-1}$.

Observe that if $K$ is a complete graph with $n$ nodes, then $A_K\cong \Z^n$; at the other end of the spectrum, if $K$ has no edges then $A_K\cong F_n$, the free group on $n$ letters. For a fixed number $n$ of nodes, the groups $A_K$ interpolate between these two extremal cases of $\Z^n$ and $F_n$.
For instance, for a complete bipartite graph $K$, $A_K$ is a direct product of two free groups, while for a not connected graph $K$, $A_K$ is the free product of the RAAGs defined by the connected components of $K$.

In this paper we will study the automorphism group $\Aut(A_K)$ of $A_K$ which, by the discussion of the paragraph  above,  interpolates between the important cases of $\Aut(F_n)$ and $\Aut(\mathbb{Z}^n)= \GL(n,\Z)$. More concretely, we will restrict our attention to the case when the defining graph $K$ is a tree.

%\nota{quiz\'as poner aqu\'{\i} ejemplos expl\'{\i}citos de $\Aut(A_K)$ para alg\'un $K$ especial}

\subsection{Abelianization of finite-index subgroups}  Recall that the {\em abelianization} of a group $G$ is the quotient $G^{\rm{ab}}:= G/[G,G]$, where $[G,G]$ is the commutator subgroup of $G$. By definition, $G^{\rm{ab}}$ is abelian, and has a further incarnation as the first homology group $H_1(G,\Z)$ of $G$.
Observe that, since~$G^{\rm{ab}}$ is abelian, it may be decomposed as $B \oplus \Z^N$, where $B$ is a finite abelian group. The number $N$ is called the {\em rank} of $G^{\rm{ab}}$, also known as the {\em first Betti number} of $G$, denoted $b_1(G)$.

A celebrated theorem of Kazhdan \cite{Kazhdan} implies that if $G<\Aut(\mathbb{Z}^n)=\GL(n,\Z)$ has finite index, then $b_1(G)=0$.
Motivated by this, a well-known open question asks whether the same holds true for finite index subgroups of $\Aut(F_n)$, where $n\ge 4$. We remark that the condition $n\ge 4$ is crucial, for Grunewald--Lubotzky \cite{GL} have constructed an explicit finite-index subgroup of $\Aut(F_3)$ with positive first Betti number.

We may consider the analogous problem for automorphism groups of arbitrary RAAGs, although one needs to be slightly careful about how to formulate it. Indeed, it is often the case that $\Aut(A_K)$ contains $A_K$ as a subgroup of finite index (see Charney--Farber \cite{CF} and Day \cite{Day-random} for explicit results in this direction), and $b_1(A_K)\ge 1$ so long as $K$ has at least one node.
With this caveat in mind, one may still search for  combinatorial conditions on $K$ that guarantee the existence of finite index subgroups of $\Aut(A_K)$ with positive first Betti number, and which apply to graphs $K$ for which~$A_K$ has infinite index in $\Aut(A_K)$ (this is the case, for instance, when $K$ is a tree).

The discussion of this type of conditions is the objective of this paper. Before proceeding any further,  we introduce some definitions and notations about graphs and, in particular, trees.

\subsection{Graphs and trees}

Given a graph $K$, the set of neighbors of a vertex $v\in V(K)$ will be called the \textit{link} of $v$:
\[
\lk(v) = \{ w \in V(K) \mid vw \in E(K)\}.
\]
The {\em degree} of $v$ is the cardinality of $\lk(v)$.  We define the {\em star} of $v$ as $\st(v) := \lk(v) \cup \{v\}$. For simplicity, by $\st(v)$ we also mean the subgraph with these vertices and the edges from $v$ in $K$.

If $K_2$ is a subgraph of a graph $K_1$, we denote by $K_1 \setminus K_2$ the full subgraph of $K_1$ induced by the vertices which belong to $K_1$ but not to $K_2$.

We endow $V(K)$ with its usual combinatorial distance $d$, namely, given $u,v\in V(K)$, we define $d(u,v)$ as the minimal $n$ for which there exist nodes $u= w_0, w_1, \ldots, w_n=v$ in $V(K)$ with $w_jw_{j+1} \in E(K)$ for all $j$.

\subsubsection{Trees}
A graph $T$ is a {\em tree} if it is connected and every edge separates~$T$ into two connected components, or, alternatively,  if it is connected and has no {\em cycles}. From now on, we will use $T$ to refer to a tree.

A node of a tree $T$ is a {\em leaf} if its degree is 1.
The {\em boundary} $\partial T$ of $T$ is the set of leaves of~$T$ and, for a node $v$ of $T$, we write $\partial_v$ to abbreviate $d(v, \partial T)$, the distance from $v$ to $\partial T$.

A node $v$ of a tree $T$ is called {\em deep} if $\partial_v\ge 3$. The subset of deep nodes of a tree $T$ is denoted by $D(T)$. A tree is termed {\em shallow} if it has no deep nodes, i.e., all its nodes are at distance at most 2 from the boundary.

We say that a tree $T$ is  {\em rooted} if it has one distinguished node, called the {\em root} of $T$. We say that $T$ is {\em labeled} if
there is a bijection between the nodes of $T$ and the set $\{1, \ldots, n\}$, where $n$ is the number of nodes of~$T$.

\subsection{Automorphisms of RAAGs defined by trees}
In \cite{AMP}, the first and fourth named authors  identified two properties of a graph $K$, each of which implies the existence of finite-index subgroups of $\Aut(K)$ with positive first Betti number; see Corollary 1.4 and Theorem 1.6 of \cite{AMP}. In this paper we will study one of these conditions in the particular case when $T$ is a tree.
At this point we remark that, apart from forming a natural subclass, RAAGs defined by trees are also interesting from a topological viewpoint, as they are examples of fundamental groups of certain three-dimensional manifolds called {\em graph manifolds}.

We shall denote by ${\Aut^\star}(A_K)$ the finite-index subgroup of $\Aut(A_K)$ generated by {\em transvections}, {\em partial conjugations}, and {\em thin inversions}; see Section~\ref{sec:RAAGS} for an expanded definition.

The following theorem is Proposition 5.3 in \cite{AMP}, which is simply a restatement of Theorem 1.6 in \cite{AMP} in the particular case when the given graph is a tree.

 \begin{theorem*}[\cite{AMP}]\it
 If the tree $T$ is not shallow, then  $b_1({\Aut^\star}(A_T)) \ge 1$.
 \label{thm:AMP}
\end{theorem*}

In this note we refine the methods of \cite{AMP} in order to give a better lower bound on this rank, again in the particular case when $T$ is a tree.

We introduce, for any tree $T$, the graph invariant
$$
\Upsilon(T):= \sum_{v\in D(T)}  \sum_{w\in \lk(v)} (\deg(w) -1)\, .
$$
As we shall see in Section \ref{sec:RAAGS}, the invariant $\Upsilon(T)$ counts precisely the number of the so-called partial conjugations of $A_T$ defined by deep nodes.

Observe that for a deep node $v$ in a tree, $\sum_{w\in \lk(v)} (\deg(w) -1)\ge 2$, and thus for any tree $T$,
\begin{equation}
\label{eq:comparison Upsilon number of deep}
\Upsilon(T)\ge 2 |D(T)|\, .
\end{equation}

The first result of this paper is the following lower bound of the Betti number of $\Aut^\star(A_T)$.
\begin{theoremletra}\it
For any tree $T$,   the bound $b_1({\Aut^\star}(A_T)) \ge \Upsilon(T)$ holds.
\label{mainthm:hom}
\end{theoremletra}

This result implies that if $T$ is a tree with at least one deep node, then $b_1({\Aut^\star}(A_T)) \ge 1$, as asserted in the result from \cite{AMP} stated above.

\subsection{Combinatorics of deep nodes}
Next we turn our attention to the combinatorial analysis of deep nodes and shallow trees, and to the study of the ``typical size'' of the combinatorial invariant $\Upsilon(T)$. %in order to respond to questions like: how frequent are shallow trees?, or what is the typical size of $\Upsilon(T)$?

We will   carry out this study in terms of {\em labeled} trees. This corresponds to considering RAAGs with labeled generators. Of course, isomorphic classes of \textit{unlabeled} trees correspond to isomorphic classes of RAAGs.

%If $T$ is a labeled tree, we will still denote by $A_T$ the RAAG defined by $T$, although it is to be observed  that the isomorphism class of a RAAG is %determined by the isomorphism class of the underlying graph  \cite{Droms}, and therefore different labeled trees will give rise to the same RAAG.

Let $\mathcal{U}_n$ denote the set of trees with $n$ nodes labeled with $\{1,\ldots, n\}$. Cayley's theorem says that the cardinality of this set $\mathcal{U}_n$ is exactly $n^{n-2}$ for~$n \ge 1$.

As we will see below (Theorem \ref{thm:betti on average como n}), for a typical tree $T$ in $\mathcal{U}_n$, $b_1({\Aut^\star}(A_{T}))$ is quite large;
%
%As we will see (Theorem \ref{thm:estimate of number of shallow trees}), the proportion in $\mathcal{U}_n$ of shallow trees tends to~$0$ exponentially fast as %$n\to \infty$. This, combined with the result from \cite{AMP} stated above,  immediately  yields the following result:
%
%First, we will show:
%
%\begin{proposition}
%The proportion in $\mathcal{U}_n$ of shallow trees tends to~$0$ exponentially fast as $n\to \infty$.
%\label{thm:shallow are aplenty}
%\end{proposition}
%Theorem \ref{thm:estimate of number of shallow trees} will provide a more precise formulation of this result.
%
%\
%
%The combination of  Theorems \ref{thm:AMP} and Proposition \ref{thm:shallow are aplenty}  yields immediately:
%
%\begin{theoremletra}\label{thm:betti exponential}\it
%The proportion in $\mathcal{U}_n$ of those trees $T\in \mathcal{U}_n$ with Betti number $b_1({\Aut^\star}(A_T))=0$ tends to $0$ exponentially fast as $n\to %\infty$.
%\end{theoremletra}
%
%Before we continue, we point out that
although we point out that for every $n$ there  are trees $T$ with~$n$ nodes such that $b_1({\Aut^\star}(A_{T})) =0$; see Lemma \ref{lemma:trees with betti null}. It seems that the proportion of trees in $\mathcal{U}_n$ for which $b_1({\Aut^\star}(A_{T})) =0$ is asymptotically negligible as $n\to\infty$.   It would be nice to ascertain this, and to establish the precise speed of convergence to $0$.

\smallskip

%Continuing with the description of our results,
Theorem \ref{thm:expected number of deep nodes in unrooted trees} below
%Even further we can prove:
%\begin{proposition}
%As $n \to \infty$, the average value among the trees in $\mathcal{U}_n$ of the proportion of nodes which are deep tends to a certain constat $c_3>0$.
%\label{thm:on average proportion c_3 of deep}\end{proposition}
%This is actually the content of Theorem \ref{thm:expected number of deep nodes in unrooted trees}.
%
%Proposition \ref{thm:on average proportion c_3 of deep} claims
asserts that
$$\lim_{n \to \infty} \frac{1}{|\mathcal{U}_n|}\sum_{T \in \mathcal{U}_n}\frac{|D(T)|}{n}=c_3\, .$$
The constant $c_3$ is about $0.35$, and thus we can say that for $n$ large a typical labeled tree with $n$ nodes has about 35\% of deep nodes.

\smallskip

Concerning the invariant $\Upsilon(T)$, as we will see in Theorem \ref{thm:star in unrooted trees}, we have that
%\begin{proposition}
%As $n \to \infty$, the average value among all trees $T \in\mathcal {U}_n$ of the quotient $\Upsilon(T)/n$ tends to a certain constant $d_3>0$
%\label{thm:average value of upsilon}
%\end{proposition}
%This is actually the content of Theorem \ref{thm:star in unrooted trees}.
%
%Proposition \ref{thm:average value of upsilon} claims that
$$
\lim_{n \to \infty} \frac{1}{|\mathcal{U}_n|}\sum_{T \in \mathcal{U}_n}\frac{|\Upsilon(T)|}{n}=d_3\, ;
$$
the value of the constant $d_3$ is $\approx 2\mbox{.}070$.
 In particular, we may say that for~$n$ large and a typical tree $T\in \mathcal{U}_n$, the invariant $\Upsilon(T)$ is about $d_3 \, n$.

In other words, we will get:

\begin{theoremletra}\label{thm:betti on average como n}\it
For $n$ large and a typical tree $T \in U_n$, we have \[b_1(\Aut^\star(A_{T})) \ge d_3 \, n.\] More concretely, \[\liminf_{n\to \infty}\frac{1}{|U_n|} \sum_{T\in U_n} \frac{b_1(\Aut^\star(A_{T}))}{n} \ge d_3.\]
\end{theoremletra}

For the sake of completeness at this point, we remark that the explicit values of the constants $c_3$ and $d_3$ are
$$
c_3=\frac{1}{e}\, e^{-1/e} \,e^{(e^{1-1/e}-1)/e}\quad\text{and}\quad
d_3=2-\frac{1}{e}+\frac{1}{e}\, \Big(1-\frac{1}{e}\Big)\,e^{1-1/e}\, .
$$

\subsection{Plan of the paper}
Section \ref{sec:RAAGS} is devoted to the proof of Theorem \ref{mainthm:hom}.
Section \ref{sec:deep and shallow}
contains the combinatorial analysis which leads to the proof %s
 of Theorem %s \ref{thm:betti exponential} and
 \ref{thm:betti on average como n}.

\subsection*{Acknowledgements} J.\,A and C.\,M. are partially supported by the MI\-NE\-CO grant MTM2015-67781-P; further, J.\,A is funded by the Ram\'on y Cajal grant 2013-RYC-13008. J.\,L.\,F. and P.\,F. are partially supported by Fundaci\'on Akusmatika, and C.\,M. is funded by Gobierno de Arag\'on and European Regional
Development Funds.

\section{RAAGs and their automorphisms}\label{sec:RAAGS}

Recall from the introduction that, given a finite graph $K$, the {\em right-angled Artin group} (RAAG, for short) defined by $K$ is the group $A_K$ with presentation \[A_K = \langle v\in V(K) \mid [v,w]=1 \iff vw \in E(K)\rangle.\]
%Observe that $A_K$ is a free group if and only if $K$ has no edges, and that~$A_K$ is free-abelian if and only if $K$ is a complete graph.   %\nota{aqu\'{\i}, $v$ y $v^{-1}$. Tambi\'en $V(K)^{-1}$}
 In order to relax notation, we will blur the distinction between nodes of $K$ and generators of $A_K$. For instance, given a vertex $v\in V(K)$ we will write $v^{-1}$ for the inverse of the generator corresponding to $v$ in $A_K$. In addition, we will write $V(K)^{-1}$ for the set of inverses of elements of $V(K)$, when viewed as generators of $A_K$.

Here we will focus on the automorphism group $\Aut(A_K)$ of $A_K$. Our first aim is to describe a standard generating set for $\Aut(A_K)$, introduced by Laurence~\cite{Laurence} and Servatius \cite{Servatius}. Before doing so, we will need to introduce a certain partial order on the set of vertices of $K$.   %\nota{\'este ser\'{\i}a el sitio para la seccioncilla con $\le$, $\sim$, y observaciones al respecto?}

\subsection{A partial order on the set of nodes}
\label{subsec:po} There is a standard partial order $\le$ on the set of nodes of $K$, whereby $v\le w$ if $\lk(v) \subset \st(w)$. We will write $v\sim w$ to mean $v\le w$ and $w\le v$; it is easy to see that $\sim$ is an equivalence relation. We will say that a node $v\in V(K)$ is {\em thin} if its equivalence class, with respect to $\sim$, has exactly one element.

We record the following observation for future use:

\begin{lemma}
Let $T$ be a tree with at least three nodes. Then
\begin{enumerate}
\item $v\le w$ if and only if $v\in \partial T$ and $d(v,w)\le 2$;
\item $v\sim w$ if and only if $v,w\in \partial T$ and there exists $u\in V(T)$ with $v,w\in \lk(u)$.
\end{enumerate}
\label{lem:po}
\end{lemma}

Note that ${\rm (ii)}$ above implies that the $\sim$-equivalence classes of nodes with $k\ge 2$ elements consist precisely of sets of $k$ leaves which are neighbors of a same node.
A further consequence is that the  thin nodes of a tree are either nodes that are not leaves, or leaves whose only neighbor is not connected to other leaves.

\subsection{Laurence--Servatius generators}
\label{subsec:LS}
We distinguish the following four types of automorphisms of $\Aut(A_K)$:
\begin{enumerate}
\item {\em Graphic automorphisms.} Every automorphism of $K$ induces an element of $\Aut(A_K)$, which we call {\em graphic}.
\item {\em Inversions.} Given $v\in V(K)$, the {\em inversion} on $v$ is the automorphism that sends $v$ to $v^{-1}$, and fixes the rest of generators.
\item {\em Transvections.} Given $u,v\in V(K)$, the {\em transvection} $t_{uv}$ sends $u$ to $uv$, and fixes the rest of generators. It is not difficult to see that $t_{uv} \in \Aut(A_K)$ if and only if $u\le v$.

\item {\em Partial conjugations.} Let $u \in V(K)$, and let $Y$ be a connected component of $K \setminus \st(u)$. The {\em partial conjugation} $c_{u,Y}$ is the automorphism given by $c_{Y,u} (v) = u^{-1} v u$ if $v\in Y$, and $c_{Y,u}(v) =v$ otherwise.
\end{enumerate}

Laurence \cite{Laurence} and Servatius \cite{Servatius} proved that these four types of automorphisms suffice to generate $\Aut(A_K)$:

\begin{theorem}[\cite{Laurence}, \cite{Servatius}]
Let $K$ be any graph. Then $\Aut(A_K)$ is generated by the sets of graphic automorphisms, inversions, transvections, and partial conjugations.
\end{theorem}

\subsection{Day's presentation of $\Aut(A_K)$} More recently, building on work of McCool \cite{McCool},  Day \cite{Day} gave an explicit finite presentation of $\Aut(A_K)$ in terms of {\em Whitehead automorphisms}, which we now briefly recall.

 Let $L = V(K) \cup V(K)^{-1} \subset A_K$, and consider the obvious extension to $L$ of the partial order $\leq$.
A {\em type {\rm(1)} Whitehead automorphism} is an automorphism of $A_K$ which is induced by a permutation of $L$.
%As this automorphism has to preserve the standard relators of $A_T$, it is essentially a graph automorphism of $T$ followed by some inversions.
A {\em type {\rm(2)} Whitehead automorphism} is determined by a subset $A\subset L$, plus an $a\in L$ with $a \in A$ but $a^{-1} \notin A$. Then we set $(A,a)(a) = a$ and, for $c \ne a$,
$$(A,a)(c)= \left \{\begin{aligned}
c, && \text{ if } c \notin A \text{ and } c^{-1} \notin A, \\
ca, && \text{ if } c \in A \text{ and } c^{-1} \notin A, \\
a^{-1}c, && \text{ if } c \notin A \text{ and } c^{-1} \in A ,\\
a^{-1}ca, && \text{ if } c \in A \text{ and } c^{-1} \in A .\\
\end{aligned}
\right.$$

We stress that not every choice of $A\subset L$ and $a \in L$ gives rise to an automorphism of $A_K$. In this direction, we have:

\begin{lemma}[\cite{Day}, Lemma {\rm2.5}]\label{welldef} Let $A\subset L$, and $a\in L$ with $a\in A$ but $a^{-1} \notin L$. Then $(A,a)\in \Aut(A_K)$ if and only if
\begin{itemize}
\item[{\rm(1)}] the set $K\cap A\cap A^{-1}\setminus \lk(a)$ is a union of connected components of $K \setminus \st(a)$,

\item[\rm{(2)}] for each $x\in A-A^{-1}$ we have $x\leq a$.
\end{itemize}
\end{lemma}

\begin{remark}{\upshape  Observe that every Laurence--Servatius generator of $\Aut(A_K)$ may be expressed in terms of Whitehead automorphisms. This is clear for graphic automorphisms and inversions, which are type (1) automorphisms.

 In the case of partial conjugations, if $Y$ is a union of connected components of $K \setminus \st(a)$,
$$c_{Y,a}=(Y\cup Y^{-1}\cup a,a)$$
and in particular
$$c_a:=c_{K-a,a}=(L-a^{-1},a)$$

Finally, if $\tau_{ba}$ is a transvection (so, in particular, $b\leq a$) then
$$
\tau_{ba}=(\{a,b\},a).
$$
}
\end{remark}

In \cite{Day}, Day proved the following.

\begin{theorem}[\cite{Day}]
$\Aut(A_K)$ is the group generated by the set of all Whitehead automorphisms, subject to the following relations:

\begin{enumerate}
\item[(R1)] $(A,a)^{-1} = (A - a \cup a^{-1}, a^{-1})$,

\item[(R2)] $(A,a)(B,a) = (A \cup B, a)$ whenever $A \cap B = \{a\}$,

\item[(R3)] $(B, b)(A,a)(B,b)^{-1}  = (A,a)$, whenever $\{a,a^{-1}\} \cap B = \emptyset$, $\{b, b^{-1} \} \cap A = \emptyset$, and at least one of $A \cap B = \emptyset$ or $ b \in \lk(a)$ holds,

\item[(R4)] $(B, b)(A,a)(B,b)^{-1} =  (A,a)(B - b \cup a, a)$, whenever $\{a,a^{-1}\} \cap B = \emptyset$, $b \notin A$, $b^{-1} \in A$,  and at least one of $A \cap B = \emptyset$ or $ b \in \lk(a)$ holds,

\item[(R5)] $(A - a \cup a^{-1}, b)(A, a) = (A - b \cup b^{-1}, a) \sigma_{a,b}$, where $b \in A$,
$b^{-1} \notin A$, $b \ne a$ but $b \sim a$, and where $\sigma_{a,b}$ is the type {\rm(1)} automorphism such that $\sigma_{a,b}(a) = b^{-1}$, $\sigma_{a,b}(b) = a$, fixing the rest of generators.

\item[(R6)] $\sigma (A,a) \sigma^{-1}= (\sigma(A), \sigma(a))$, for every $\sigma$ of type {\rm(1)}.
\item[(R7)] All the relations among type {\rm(1)} Whitehead automorphisms.

\item[(R9)] $(A,a)(L - b^{-1}, b) (A,a)^{-1} = (L - b^{-1}, b)$, whenever $\{b,b^{-1}\} \cap A = \emptyset$, and

\item[(R10)]$(A,a)(L - b^{-1}, b) (A,a)^{-1}= (L - a^{-1}, a)( L - b^{-1}, b)$, whenever $b \in A$ and $b^{-1} \notin A$.
\end{enumerate}
\label{thm-day}
\end{theorem}

\begin{remark}
{\upshape In Day's list of relations \cite{Day} there is an extra type of relator, which Day calls (R8); however, as he mentions in \cite{Day}, Remark 2.7, this relation is redundant and therefore we omit it from the list above.}
\end{remark}

\subsection{The group ${\Aut^\star}(A_K)$} From now on we will restrict our attention to an explicit finite-index subgroup of $\Aut(A_K)$, which we will denote by ${\Aut^\star}(A_K)$. Before introducing this subgroup, we need a definition.
Recall that a node is said to be thin if its equivalence class, with respect to the relation $\sim$, has only one element. Consequently, we call an inversion {\em thin} if it fixes every thin node; in other words, it is the inversion about a node that is not thin.

%\nota{graph theoretical meaning of ``thin''? A la secci\'on de $\sim$}

Now, let ${\Aut^\star}(A_K)$ be the subgroup of $\Aut(A_K)$ generated by transvections, partial conjugations, and thin inversions. Observe that ${\Aut^\star}(A_K)$ has finite index in $\Aut(A_K)$.

In \cite{AMP}, the first and fourth named authors proved that $\Aut(A_K)$ has a finite-index subgroup that surjects onto $\Z$. In that paper, it was claimed that one such finite-index subgroup is the one generated by transvections, partial conjugations, and {\em all} inversions, which was denoted $\Aut^0(A_K)$. However, the proof given in \cite{AMP} is not correct; this issue was fixed in an updated version of \cite{AMP} (see \cite{AMP-erratum}), where it was proved that $\Aut^\star(A_K)$ surjects to $\Z$.   In order to do so, one needs to prove  that Day's presentation can be restricted in the obvious way to a presentation for ${\Aut^\star}(A_K)$. For completeness, we include a proof here.
 Write  $\operatorname{Sym}^1(A_K)$ for the subgroup of $\Aut(A_K)$ consisting of those graphic automorphisms that preserve setwise the equivalence classes for $\sim$ and fix every node of $K$ that is thin. One has:

\begin{proposition} The group ${\Aut^\star}(A_K)$ has a finite presentation with generators the set of type {\rm(2)} Whitehead automorphisms and $\operatorname{Sym}^1(A_K)$, and relators {\rm(R1), (R2), (R3), (R4), (R5), (R9), (R10)} above together with
\begin{enumerate}
\item[{\rm(R6)'}] $\sigma (A,a) \sigma^{-1}= (\sigma(A), \sigma(a))$, for every $\sigma\in\operatorname{Sym}^1(A_K)$.
\item[{\rm(R7)'}] All the relations among automorphisms in $\operatorname{Sym}^1(A_K)$.
\end{enumerate}
\end{proposition}

%
%
%
%From now on we will restrict our attention to the subgroup of $\Aut(A_T)$  generated by partial conjugations, transvections and some of the inversions. More precisely, we allow only those inversions $a\mapsto a^{-1}$ such that $a\sim b$ for some other $b$, in other words, so that the equivalence class $[a]$ of the equivalence relation $\sim$ in $T$ consists of at least 2 elements.  This group also contains the type (1) Whitehead automorphisms that preserve setwise the equivalence classes in $T$ up to possible inversion except of the case of single element classes (\cite{Day}). Note that ${\Aut^\star}(A_T)\leq\Aut^0(A_T)$ where this is the subgroup generated by partial conjugations, transvections and all inversions. In \cite{AMP}  we argued that the above relators with the obvious restrictions in (R6) and (R7) give a presentation of $\Aut^0(A_T)$, now we claim that essentially the same argument implies the same for the group ${\Aut^\star}(A_T)$, we denote by $R^0$ and $R^1$ the resulting set of relators. Then $R^0$ and $R^1$ consist  of the whole list (R1)-(R10) except for possibly some relators of type (R6) and (R7).
%{\sl I copy  paste here the whole argument with the needed modifications, later we can think about how to explain this.
%

\begin{proof}
First, it follows directly from  the definition that $\Aut^\star(K)$ is generated by all the type (2) Whitehead automorphisms, and every thin inversion. Thanks to relator (R5), we may add the elements of $\operatorname{Sym}^1(A_K)$ to this list of generators.

Let $R^1$ be the list (R1)--(R10) of Day's relators, except that (R6) and~(R7) are substituted by (R6)' and (R7)'. Observe that every relation in $R^1$ is indeed a relation in ${\Aut^\star}(A_K)$. Therefore, it remains to justify why these form a complete set of relations in ${\Aut^\star}(A_K)$.

By Theorem A.1 of \cite{Day}, every automorphism $\alpha \in \Aut(A_K)$ may be written as a product $\alpha= \beta \delta$, where $\beta$ lies in the subgroup of  $\Aut(A_K)$ generated by {\em short-range} automorphisms, and $\delta$ is in the subgroup generated by {\em long-range} automorphisms.  Here, we say that $\gamma \in \Aut(A_K)$ is {\em long-range} if either it is a type (1) Whitehead automorphism, or it is a type (2) Whitehead automorphism  specified by a subset $(A,v)$ such that $\gamma$ fixes all the elements adjacent to~$v$ in~$K$. Similarly, we say that  $\gamma \in A_K$ is {\em short-range} if it is a type (2) Whitehead automorphism specified by a subset $(A, v)$ and $\gamma$ fixes all the elements of $K$ not adjacent to $v$. Following Day, we denote by $\Omega_l$ (resp. $\Omega_s$) the set of all long-range (resp. short-range) automorphisms.

Consider now $\alpha \in {\Aut^\star}(A_K)$, and observe that all short-range automorphisms are in ${\Aut^\star}(A_K)$. The proof of the splitting in Theorem A.1 of \cite{Day} is based in the so called {\em sorting substitutions} in \cite{Day}, Definition 3.2. Of these, only substitution (3.1) involves an element possibly not in ${\Aut^\star}(A_K)$, and this element is just moved along, meaning that if our initial string consists  solely of elements in ${\Aut^\star}(A_K)$, then so does the final  string. Moreover, observe that the relators needed for these moves all lie in $R^1$ (an explicit list of the relators needed, case by case, can be found in Lemma~3.4 of~\cite{Day}). All this implies that up to conjugates of relators in $R^1$, we may write $\alpha= \beta \delta$, with  $\beta$ in the subgroup of  ${\Aut^\star}(A_K)$ generated by $\Omega_s$, and $\delta$  in the subgroup generated by  $\Omega_l^1=\Omega_l \cap {\Aut^\star}(A_K)$.

By Proposition 5.5 of \cite{Day}, the subgroup of ${\Aut^\star}(A_K)$ generated by $\Omega_s$ has a presentation whose every generator is a short-range automorphism or an element of $\operatorname{Sym}^1(A_K)$, and whose every relator lies in $R^1$. Indeed, in the proof of Proposition~5.5 in~\cite{Day}, the generators that we need to add to $\Omega_s$ to get the desired presentation are precisely the elements of the form $\sigma_{ab}$ of~(R5), which belong to $\operatorname{Sym}^1(A_K)$.

In addition, the subgroup ${\Aut^\star}(A_K)$ generated by $\Omega^1_l $ has a presentation whose every relator is in $R^1$. To see that this is indeed the case, first recall from Proposition 5.4 of \cite{Day} that the subgroup of $\Aut(A_K)$ generated by $\Omega_l$ admits a presentation in which every relation (also in the list (R1)--(R10) of Theorem \ref{thm-day}) is written in terms of $\Omega_l$. In order to prove this, Day uses a certain inductive argument called the {\em peak reduction algorithm}. However, by Remark 3.22 of \cite{Day}, every element of ${\Aut^\star}(A_K)$ may be peak-reduced using elements of ${\Aut^\star}(A_K)$ {\em only}. Indeed, the only subcase of Remark 3.22 in \cite{Day} that is problematic in this setting is the use of subcase~(3c) of Lemma~1.18 in~\cite{Day}. But the relator used in that subcase is precisely (R5), where the type~(1) Whitehead automorphism is $\sigma_{ab}$, and thus lies in $\operatorname{Sym}^1(A_K)$.

Moreover, the process of peak reduction needs relators in $R^1$ only; this is a consequence of the fact, observed already in Remark 3.22 of \cite{Day}, that  type~(1) Whitehead automorphisms are only moved around when lowering peaks, and if they lie in $\Omega_l^1$ then the needed relator is precisely (R5), where the type~(1) Whitehead automorphism is $\sigma_{ab}$ and thus lies in $\operatorname{Sym}^1(A_K)$.
\end{proof}

\subsection{Proof of Theorem \ref{mainthm:hom}}

In what follows we will assume that $T$ is a tree with at least 3 nodes. Recall that $\partial T$ denotes the set of leaves of $T$, that is, the set of nodes of degree one.
Before embarking in the proof of Theorem~\ref{mainthm:hom}, we make some preliminary observations.
%\nota{a la secci\'on con~$\sim$}

First, an immediate consequence of Lemma \ref{lem:po} is that if $a$ is a deep node of $T$, then there is no transvection of the form $\tau_{ca}$. Furthermore, recall that the same lemma implies that the $\sim$-equivalence classes with more than one element consist precisely of sets of $k\ge 2$ leaves adjacent to a same node. Thus the subgroup of $\Aut(A_T)$ whose elements are those graphic automorphisms which (setwise) preserve these classes  is generated by the graphic automorphisms that fix the whole~$T$, apart from two leaves adjacent to a same node, which are possibly interchanged by an involution.

%\begin{lemma} If $b\leq a$, then $b\in \partial T$ and $1\le d(a,b) \le 2$.
%\end{lemma}
%
% As an immediate consequence, if $a$ is a deep node, then there is no transvection of the form $\tau_{ca}$. Furthermore, given $a,b\in V(T)$, we have $a\sim b$ \nota{a la secci\'on con~$\sim$} if and only if $a,b\in \partial T$ and both are adjacent to the same node. In particular, the equivalence classes for $\sim$ consist precisely of sets of leaves adjacent to a same node, and thus the subgroup of $\Aut(A_T)$ whose elements are those graphic automorphisms which (setwise) preserve these classes  is generated by the graphic automorphisms that fix the whole~$T$, apart from two leaves adjacent to a same node, which are interchanged up to a possible inversion.

Let $a\in V(T)$. Observe that the number of partial conjugations of the form $c_{Y,a}$ coincides with the number of connected components of $T \setminus \st(a)$.
 Moreover, this number can be computed as
$$
\sum_{c\in\lk(a)}(\mathrm{degree}(c)-1).
$$
Set \[\Omega=\bigcup\{ c_{Y,a}\mid  d(a,\partial T)\ge 3, Y\text{ connected component of }T \setminus \st(a)\}\}.\]
Finally, in order to relax notation, we will simply write $H_1$ instead of $H_1({\Aut^\star}(A_T), \Z)$. After all this notation, Theorem \ref{mainthm:hom} will be a consequence of the following stronger result.

\begin{theorem}
Let $\pi\colon {\Aut^\star}(A_T) \to H_1$ be the abelianization map. Then $\pi(\Omega)$ is a linearly independent set in $H_1$.
\label{prop:H_1}
\end{theorem}

Accepting momentarily the validity of Theorem \ref{prop:H_1}, we now explain how to deduce Theorem {\rm\ref{mainthm:hom}} from it:

\begin{proof}[Proof of Theorem {\rm\ref{mainthm:hom}}]
In the light of the discussion before Theorem \ref{prop:H_1}, we have that the cardinality of $\Omega$ is equal to
\[\Upsilon(T)=\sum_{v\in D(T)} \sum _{w\in \lk(v)} (\deg(w) -1),\] where again $D(T)$ denotes the set of deep nodes of $T$. Thus the result follows from Theorem~\ref{prop:H_1}.
\end{proof}

Finally, we prove Theorem \ref{prop:H_1}:

\begin{proof}[Proof of Theorem {\rm \ref{prop:H_1}}] Let $n$ be the cardinality of $\Omega$ and consider the map
$$\begin{aligned}
\varphi:\Omega&\to\bigoplus_{c\in\Omega}\Z\\
c&\mapsto 1_c.\\
\end{aligned}$$
We claim that this map can be extended to a well defined epimorphism ${\Aut^\star}(A_T)\twoheadrightarrow\Z^{|\Omega|}$. To show this, we will first extend $\varphi$ to the set of Whitehead automorphisms that generate ${\Aut^\star}(A_T)$ and then check that Day relators are preserved. In order to do so, we map all automorphisms in $\operatorname{Sym}^1(A_T)$ to 0.

Consider an arbitrary type (2) Whitehead automorphism $(A,a)$.
If there is some leaf $b$ such that $d(a,b)\leq 2$, then we map $(A,a)\mapsto 0$. Otherwise, assume first that $a$ is a node of $T$. Using Lemma \ref{welldef} and relators (R2) we may write $(A,a)$ as a product of partial conjugations $c_{Y,a}\in\Omega$ (observe that there is no element $b\leq a$) and the set of possible $Y's$ appearing in this expression is uniquely determined from $A$. We define the image of $(A,a)$ in the obvious way using this expression; note that the last observation implies that this is well defined. Finally, in the case when $a^{-1}\in T$, set $\varphi(A,a)=-\varphi(A-a\cup a^{-1},a^{-1})$.
 Now we have an extended map which we also denote $\varphi$ and claim that it respects Day relators.  We do not have to worry about (R1) and (R2) because of the way $\varphi$ is defined. About (R3) and (R9), they are preserved because $\Z^n$ is abelian. For (R6)' and (R7)' we only have to consider elements in $\operatorname{Sym}^1(A_T)$.  Relator (R7)' is not an issue either, because all the terms therein vanish.
  So we are left with (R4), (R5) and (R10) and (R6)'. About (R4), as $\Z^n$ is abelian we only have to check that $\varphi$ maps $(B-b\cup a,a)$ to 0, but this is obvious because the facts that $b\not\in A$, $b^{-1}\in A$ and that $(A,a)$ is well defined imply $b\leq a$, hence $b$ is a leaf and $d(a,b)\leq 2$.  Exactly the same argument works for (R5) and (R10): in the case of (R5) we have $a\sim b$, thus both are leaves and everything is mapped to $0$. And in the case of (R10), we know that $b,b^{-1}\in A$, and that $(A,a)$ is well defined; thus $b\leq a$, and we argue as before to conclude that $(L-a^{-1},a)$ is mapped to 0.

  At this point, we only have to consider (R6)'. We claim that if $a$ is a deep node, and $(A,a)$ is well defined, then $(\sigma(A),\sigma(a))=(A,a)$ for any $\sigma\in\operatorname{Sym}^1(A_T)$; note that this will imply that $\varphi$ preserves (R6)'. In fact, it suffices to show the claim for $A=Y\cup Y^{-1}\cup a$ and $Y$ a connected component of $T \setminus \st(a)$. As $T$ is a tree, such a $Y$ must have more than one element and must itself be a tree with a node linked to $a$ that we can see as its root. Moreover, if $c\sim b$ are leaves in $T$ and one of them happens to be in $Y$ then so is the other. Therefore $\sigma(Y\cup Y^{-1})=Y\cup Y^{-1}$. On the other hand, since $a$ is thin we have that $\sigma(a)=a$, by the definition of    $\operatorname{Sym}^1(A_T)$, so the claim follows.
\end{proof}

%
%
%
%\begin{remark} Assume that $a$ is a node which is not a leave but it is at distance less than 3 from some leaves. Let $c_{Y,a}$ be a partial conjugation such that $Y$ consists of more than one element, in other words $Y$ is not just a leave. Then I think that the previous result can be generalized to the set $\Omega$ obtained by throwing in also this kind of partial conjugations.
%\end{remark}

\subsection{A remark on the bound given by Theorem \ref{mainthm:hom}} Before continuing, we stress that the lower bound given by Theorem \ref{mainthm:hom} is most definitely not sharp. On the other hand, not every element of ${\Aut^\star}(A_T)$ projects to a non-trivial element of $H_1$. In this direction, we have:

\begin{lemma}\label{lem:trivialH1-1} Let $T$ be a tree, and $\Aut^*(A_T)\to H_1$ the abelianization map. The following elements have trivial image:
\begin{itemize}
\item[i)] Every transvection $t_{da}=(\{d,a\},a)$ satisfying that: \begin{itemize}
\item either $a$ is a leaf, and there is a third leaf $b\notin \{a,d\}$ such that $a,b,d$ have a common neighbor.

\item $d$ is adjacent to $a$, and there is a leaf $b\ne d$ adjacent to $a$.

\item[ii)] Partial conjugations $c_{Y,a}$ where $a$ is a leaf and there is a second leaf $b\neq a$ such that $a,b$ have a common neighbor.
\end{itemize}
\end{itemize}
\end{lemma}

In order to prove the lemma, we will mainly use relators (R4) and (R10). It will be useful to reformulate them as follows (we emphasize that this reformulation does not make use of the hypothesis that $T$ is a tree).

\begin{enumerate}
\item[(R4)] Let $B_1\subseteq L$ be such that $(B_1,a)$ is well defined. %(thus $a\in B_1$, $a^{-1}\not\in B_1$).
Assume that there is some $b\in L$ with $b\leq a$ and $b,b^{-1}\not\in B_1$ such that $(B_1-a\cup b, b)$ is well defined %for this we need $b^{-1}$ not in $B_1$
 and that for some $A\subseteq L$ we have $(A,a)$ well defined,  $b \notin A$, $b^{-1} \in A$, %and for this we need $b\leq a$
  and at least one of $A \cap B_1 = \{a\}$ or $ b \in \lk(a)$ holds, Then
  $$(B_1,a)\text{ vanishes in }H_1.$$

%\item[(R5)] Let $a\neq b$ both in $L$ with $a \sim b$ and $A\subseteq L$ such that $a,b \in A$,
%$a^{-1},b^{-1} \notin A$ and the three of  $(A - a \cup a^{-1}, b)$, $(A, a)$ and  $(A - b \cup b^{-1}, a)$ are well defined. Then
%$$(A - a \cup a^{-1}, b)(A, a) = (A - b \cup b^{-1}, a) \sigma_{a,b}$$
%where $\sigma_{a,b}$ is the type (1) automorphism such that $\sigma_{a,b}(a) = b^{-1}$, $\sigma_{a,b}(b) = a$, fixing the rest of %generators.

\item[(R10)] Let $b,a\in L$ such that $b\leq a$. Then
$$c_a\text{ vanishes in }H_1,$$ where $c_a$ denotes conjugation (of every node of $T$) by $a$.
\end{enumerate}

We are now ready to prove Lemma \ref{lem:trivialH1-1}:

\begin{proof}[Proof of Lemma {\rm\ref{lem:trivialH1-1}}] First, note that since $t_{da}$ is defined, then $d$ is necessarily a leaf by Lemma \ref{lem:po}.
Moreover, in both cases we have $b\leq a$, and thus the element $(A,a)$, with $A=\{b^{-1},a\}$, is well defined. Now, in case i) let $B_1=\{d,a\}$ so $t_{da}=(B_1,a)$. As the hypothesis implies $d\leq b$, we see that  $(B_1-a\cup b, b)=(\{d,b\},b)$ is well defined, thus using (R4) we deduce that $t_{da}$ vanishes in $H_1$.

Consider now case ii). Let
$$
T-\st(a)=\{b\}\sqcup Y_1\sqcup\cdots\sqcup Y_t
$$
be the partition of $T \setminus \st(a)$ into connected components. Observe that the connected components of $T \setminus \st(b)$ are precisely
$$
T-\st(b)=\{a\}\sqcup Y_1\sqcup\cdots\sqcup Y_t
$$
also.
For any $i$, set
$B_1=Y_i\cup Y_i^{-1}\cup a$ and as before $A=\{b^{-1},a\}$. Using (R4) we deduce that $c_{Y_i,a}$ vanishes in $H_1$. Moreover, the fact that $b\leq a$ implies by (R10) that $c_a$ also vanishes in $H_1$, and as an iterated use of (R2) implies
$$c_a=c_{b,a}\prod c_{Y_i,a},$$
we see that the same happens for $c_{b,a}$.
\end{proof}

As a consequence, we may easily exhibit a class of trees $T$ for which the first Betti number of $\Aut^\star(A_T)$ vanishes.
\begin{lemma}\label{lemma:trees with betti null} Let $T$ be a tree such that every node is either a leaf or it has at least three leaves as neighbors. Then $b_1(\Aut^\star(A_T))=0$.
\end{lemma}

\begin{proof} Recall that a consequence of Day's presentation is that $\Aut^\star(A_T)$ is generated by certain type (1) Whitehead automorphisms, which have finite order, and the following two kinds of type (2) Whitehead automorphisms:
\begin{itemize}
\item[i)] Transvections $\tau_{y,a}=(\{y,a\},a)$ with $y\leq a$,

\item[ii)] Partial conjugations $c_{Y,a}=(Y\cup Y^{-1}\cup a,a)$ with $Y$ a connected component of $T-\st_T(a)$.
\end{itemize}
Therefore it suffices to check that both types of elements i) and ii) vanish in~$H_1$. In case i) this follows from the hypothesis and Lemma
\ref{lem:trivialH1-1}. The same happens in case ii) unless $a$ is not a leave. But then take a leaf $b$ that is adjacent to $a$, and $z\in\st(a)$ the node that connects $a$ to $Y$. Observe that in a similar way as we did in Lemma \ref{lem:trivialH1-1}, putting $A=\{b^{-1},a\}$ and $B_1=Y\cup Y^{-1}\cup z\cup z^{-1}\cup a$ relator (R4) implies that $c_{Y,a}$ vanishes in $H_1$. (Note that $(B_1,a)$ and $(Y\cup Y^{-1}\cup a,a)$ both represent the element $c_{Y,a}$ but we need the first one to ensure that $(B_1-a\cup b,b)$ is well defined).
\end{proof}

Before closing this section, we briefly discuss an example of a type of tree $T$ such that $\Aut(A_T)$ has infinitely many finite-index subgroups with zero Betti number. Specifically, suppose $T$ contains a vertex with degree $n$, and $n$ leaves as neighbors. Then $A_T=\Z\times F_n$, where the $\Z$-factor is generated by the vertex of degree $n$. This group satisfies properties (B1) y (B2) in \cite{AMP}, and thus, by Theorem~1.1 in that paper, we obtain that $b_1(H)=0$ for any $H\leq\Aut(A_T)$ of finite index containing the Torelli subgroup.

In the light of these results, %Proposition \ref{prop:H_1} and Lemma \ref{lem:trivialH1-1},
a natural question is:

\begin{question}
Let $T$ be a tree. What is the exact value of
$
b_1({\Aut^\star}(A_T))?
$
\end{question}

\section{Deep nodes and shallow trees}\label{sec:deep and shallow}

Recall that a node $v$ of a tree $T$ is called  deep if $\partial_v\ge 3$, that the collection of deep nodes of $T$ is denoted $D(T)$,  and that a tree $T$ with no deep nodes is termed  {shallow}.

Some examples of shallow  trees and trees with deep nodes follow; colors indicate distance to the boundary.
\begin{center}
\resizebox{12.6cm}{!}{\includegraphics{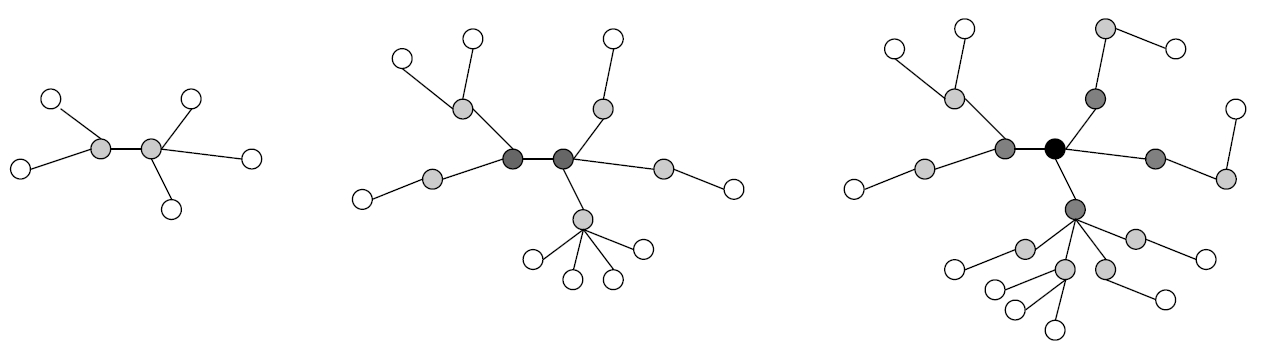}}
\end{center}

The class of rooted labeled trees is denoted by $\mathcal{T}$, while the class of general (unrooted) labeled trees is denoted by $\mathcal{U}$. The respective subclasses of trees with nodes labeled with $\{1, \dots, n\}$ are denoted by $\mathcal{T}_n$ y $\mathcal{U}_n$, for each $n \ge 1$. Cayley's theorem tells us that $$t_n:= |\mathcal{T}_n|=n^{n-1}\,,\quad \mbox{for $n \ge 1$}\, ,$$ and that $$u_n:= |\mathcal{U}_n|=n^{n-2}\,,\quad \mbox{for $n \ge 1$}\, .$$

We endow $\mathcal{U}_n$ with the uniform probability distribution; claiming that a certain property occurs with probability $p$ in $\mathcal{U}_n$ is tantamount to claiming that the proportion of trees in $\mathcal{U}_n$ satisfying that property is $p$.

\subsection{Notation and some basic results}%{Symbolic method}
\label{sect:symbolic method}

%The symbolic method readily  translates questions about counting trees into analytic questions about power series. We introduce here some basic notation and %results, and refer to the treatise of Flajolet and Sedgewick~\cite{FS} for details.

%A sequence $(a_n)_{n=0}^\infty$ of complex numbers is said to be of \textit{exponential order~$K^n$}, abbreviated as $a_n\bowtie K^n$, if
%$$
%\limsup_{n\to\infty} |a_n|^{1/n}=K.
%$$

Given a sequence $(a_n)_{n=0}^\infty$, its \textit{$($ordinary$)$ generating function} (ogf, for short) is the power series $f(z)$ given~by
$$
f(z)=\sum_{n=0}^\infty a_n\, z^n
$$
for all $z\in \mathbb{D}(0,\varepsilon)$, for some $\varepsilon >0$. We will write $a_n=\textsc{coef}_n(f(z))$.

The function $g(z)$ is the
\textit{exponential generating function} (for short, egf) of the sequence~$(a_n)$ if
$$
g(z)=\sum_{n=0}^\infty \frac{a_n}{n!}\, z^n
$$
for all $z\in \mathbb{D}(0,\varepsilon)$, for some $\varepsilon >0$.

\smallskip

A basic tool for handling combinatorial questions about trees is the Lagrange inversion formula.

\begin{lem}[Lagrange inversion formula]\label{lema:Lagrange inversion} Let $h(z)$ and $f(z)$ be two holomorphic functions on some neighborhood of $z=0$, say $\mathbb{D} (0,\varepsilon )$, such that $f(0)\not =0$, and
$$
h(z)= z f(h(z))
$$
in $\mathbb{D}(0,\delta )$.
Then, for any function $g$ holomorphic at~$0$,
$$
\mbox{\textsc{coef}}_n\left[g(h(z))\right]=\mbox{\textsc{coef}}_{n-1}\Big[g'(z)
\frac{f(z)^n}{n}\Big]\, , \qquad \mbox{for each $n \ge 1$}\,.
$$
\end{lem}
Note that $h(0)=0$.

\

\noindent {\bf Trees and generating functions.} We let $T(z)$ denote the egf of the sequence $(t_n)$, namely:
\begin{equation}\label{eq:Cayley rooted egf}
T(z)=\sum_{n=0}^\infty \frac{t_n}{n!}\, z^n,\quad |z|<\frac{1}{e}.
\end{equation}
Cayley's formula says that $T$ satisfies the following implicit equation:
\begin{equation}\label{eq:Cayley formula}
T(z)=z\, e^{T(z)}.
\end{equation}
For the class of  unrooted trees $\mathcal{U}$, we denote its egf~by
\begin{equation}\label{eq:Cayley unrooted egf}
U(z)=\sum_{n=0}^\infty \frac{u_n}{n!}\, z^n,\quad |z|<\frac{1}{e}.
\end{equation}

As an immediate corollary of Lemma \ref{lema:Lagrange inversion} we state:
\begin{cor}\label{cor:coefn of Cayleytoj}
For $n,k\ge 1$,
$$
\textsc{coef}_n\big[T(z)^k\big]=\frac{k}{n}\, \frac{n^{n-k}}{(n-k)!}.
$$
\end{cor}

%\begin{proof}
%Using \eqref{eq:Cayley formula} and the Lagrange inversion formula, lemma \ref{lema:Lagrange inversion}, we get
%\begin{equation*}
%\textsc{coef}_n\big[T(z)^k\big]=\textsc{coef}_{n-1}\Big[k\, z^{k-1}\, \frac{e^{nz}}{n}\Big]=\frac{k}{n}\textsc{coef}_{n-k}\big[e^{nz}\big]
%=\frac{k}{n}\, \frac{n^{n-k}}{(n-k)!}.\qedhere
%\end{equation*}
%\end{proof}

\

\noindent {\bf Stirling numbers.} Write $S(n,k)$ for the (double) sequence of the Stirling numbers of the second kind. We shall use the following identities.
For $k\ge 1$,
\begin{equation}\label{eq:exp-sum of S(n,k) in n}
\sum_{n=0}^\infty \frac{S(n,k)}{n!}\, x^n = \frac{1}{k!} \, (e^x-1)^k.
\end{equation}
Notice that
\begin{equation}
\label{eq:product of reciprocal of factorials}
\frac{S(n,k)}{n!}=\frac{1}{k!}\ \sum_{\substack{q_1,\dots, q_k\ge 1\\q_1+\cdots +q_k=n}}\frac{1}{q_1!\cdots q_k!}.
\end{equation}
%(just write the Taylor series of $(e^x-1)^k$).

%For $n\ge 1$,
%\begin{equation}\label{eq:sum of S(n,k) in k}
%\sum_{k=0}^\infty {S(n,k)}\, y^k = T_n(y),
%\end{equation}
%where $T_n(y)$ is a Touchard polynomial ($T_n(0)$ for all $n\ge 1$). Finally,

Also,
\begin{equation}\label{eq:double-sum of S(n,k)}
\sum_{k,n\ge 0} S(n,k)\,  \frac{x^n}{n!}\, y^k = e^{y(e^x-1)}.
\end{equation}
Taking a derivative with respect to $x$ in \eqref{eq:double-sum of S(n,k)}, we get
\begin{equation}\label{eq:double-sum of S(n,k)1dev}
\sum_{k\ge 0,n\ge 1} S(n,k) \, \frac{x^{n-1}}{(n-1)!}\, y^k = y\,e^x\, e^{y(e^x-1)};
\end{equation}
and multiplying by $x$ and differentiating again with respect to $x$,
\begin{equation}\label{eq:double-sum of S(n,k)2dev}
\sum_{k\ge 0,n\ge 1} S(n,k) \, n\, \frac{x^{n-1}}{(n-1)!}\, y^k = y\,e^x\, e^{y(e^x-1)}\, (1+x+xy\, e^x).
\end{equation}

\subsection{Deep nodes}

Our objective now is to  study how abundant  deep nodes are in a typical labeled tree with $n$ nodes, as $n \to \infty$. Our argument starts analyzing rooted trees (sections \ref{section:distance 3 from the root} and \ref{section:mean of the sum}), and then settles (section~\ref{section:from rooted to unrooted}) the same question about unrooted trees, which is the more relevant case for our purposes.

\subsubsection{Proportion of rooted trees with the root at distance $\ge 3$ to the border}\label{section:distance 3 from the root}
Recall that $\mathcal{T}$ denotes the class of all rooted trees and that $\mathcal{T}_n$ denotes the subclass of rooted trees labeled with $\{1, \dots, n\}$.
Again, we endow $\mathcal{T}_n$ with the uniform probability distribution. Probabilities and expectations, denoted by $\P_n$ and $\E_n$, refer to this probability space. Recall that $t_n=|\mathcal{T}_n|=n^{n-1}$.

Call $\mathcal{T}^{(3)}$ the subclass of rooted trees whose root is a deep node, $\partial_{\text{root}}\ge 3$. In such trees,
%Now we consider (rooted labeled) trees on $n$ nodes in which the root is, at least, at distance~3 from any leaf ($\partial_{\text{root}}\ge 3$). Say that
the root has, say, $k\ge 1$ descendants, which in turn have $q_1,\dots, q_k\ge 1$ descendants, none of which is a leaf (this guarantees distance $\ge 3$ from the root to the leaves). Call $N=q_1+\cdots +q_k$.

\begin{figure}[h]
\centering
\resizebox{8cm}{!}{\includegraphics{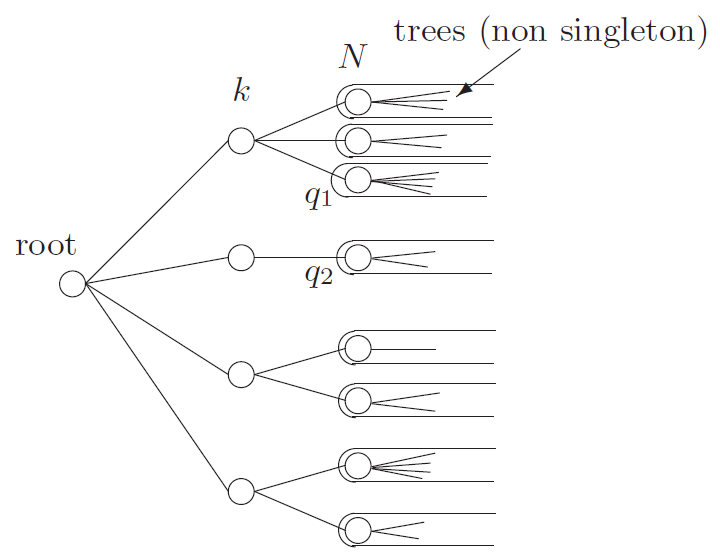}}
\caption{Trees with $\partial_{\rm root}\ge 3$.}\label{fig:tree3}
\end{figure}

Consider the following subclasses of $\T$:
\begin{align*}
\mathcal{T}_{k,q_1,\dots, q_k}^{(3)}&=\Big\{\text{\begin{tabular}{c}
$k$ descendants of the root, with $q_1,\dots, q_k\ge 1$
\\
descendants, respectively, and $\partial_{\text{root}}\ge 3$\end{tabular}}\Big\}
\\
\T_{k,N}^{(3)}&=\Big\{\text{\begin{tabular}{c}
$k$ descendants of the root, $N$ nodes
\\
in the second generation, $\partial_{\text{root}}\ge 3$\end{tabular}}\Big\}
\\&
=\bigcup_{\substack{q_1,\dots, q_k\ge 1\\q_1+\cdots+q_k=N}} \T_{k,q_1,\dots, q_k}^{(3)}.
\end{align*}
Observe that
$$
\mathcal{T}^{(3)}=\bigcup_{k\ge 1} \bigcup_{N\ge k} \T_{k,N}^{(3)}.
$$
In all cases, an extra subindex $n$ would indicate the corresponding subclass of trees with nodes labeled with $\{1,\dots,n\}$.

We have the following asymptotic result.
\begin{thm}\label{thm:prob of delta3}
$$
\lim_{n\to\infty} \mathbf{P}_n(\T_n^{(3)})=\frac{1}{e}\, e^{-1/e} \,e^{(e^{1-1/e}-1)/e}=: c_3.
$$
\end{thm}

%The numerical value of this constant $c_3$ is $\approx 0\mbox{.}3522$. Theorem \ref{thm:prob of delta3} means that, for~$n$ large, in approximately 35\% of the %trees the root is a deep node.\nota{quitar? ya aparece}

\begin{proof}
The egf of the class $\mathcal{T}_{k,q_1,\dots, q_k}^{(3)}$ is
$$
z\, \frac{z^k}{k!}\, \frac{(T(z)-z)^{q_1}}{q_1!}\cdots \frac{(T(z)-z)^{q_k}}{q_k!}=\frac{z^{k+1}}{k!}\, (T(z)-z)^N\, \frac{1}{q_1!\cdots q_k!},
$$
and so,
$$
\P_n(\T_{n;k,q_1,\dots, q_k}^{(3)})=\frac{|\T_{n;k,q_1,\dots, q_k}^{(3)}|}{n^{n-1}}=\frac{n!}{n^{n-1}}\ \frac{1}{q_1!\cdots q_k!}\, \frac{1}{k!}\ \textsc{coef}_n[z^{k+1}\, (T(z)-z)^N].
$$

Now, writing $a_j=(k+1)+(N-j)$ and using corollary \ref{cor:coefn of Cayleytoj},
\begin{align*}
\textsc{coef}_n\big[z^{k+1}\, (T(z)-z)^N\big]&=\textsc{coef}_{n-k-1}\Big[\sum_{j=0}^N {N\choose j} T(z)^j \, (-1)^{N-j}\, z^{N-j}\Big]
\\
&=\sum_{j=0}^N {N\choose j}  (-1)^{N-j}\, \textsc{coef}_{n-a_j}\big[T(z)^j\big]
\\
&=\sum_{j=0}^N {N\choose j}  (-1)^{N-j}\, \frac{j}{n-a_j}\, \frac{(n-a_j)^{n-a_j-j}}{(n-a_j-j)!}.
\end{align*}
This yields
\begin{align}\nonumber
&\P_n(\T_{n; k,q_1,\dots, q_k}^{(3)})
\\ \label{eq:Pn(Akq1...qk3)}
&\qquad =\frac{n!}{n^{n-1}}\,\frac{1}{q_1!\cdots q_k!}\, \frac{1}{k!}\ \sum_{j=0}^N {N\choose j}  (-1)^{N-j}\, \frac{j}{n-a_j}\, \frac{(n-a_j)^{n-a_j-j}}{(n-a_j-j)!}.
\end{align}

Notice that
\begin{align}\nonumber
\frac{n!}{n^{n-1}}\, \frac{1}{n-a_j}\, &\,\frac{(n-a_j)^{n-a_j-j}}{(n-a_j-j)!}
\\
\label{eq:cuenta para Pn(Akq1...qk3)}
&=\frac{n}{n-a_j}\, \frac{n(n-1)\cdots (n-a_j-j+1)}{n^{a_j+j}}\Big(1-\frac{a_j}{n}\Big)^{n-a_j-j},
\end{align}
which tends to $e^{-a_j}$ when $n\to\infty$.

This gives, recalling that $a_j=(k+1)+(N-j)$, that
\begin{align*}
\lim_{n\to\infty} \P_n(\T_{n; k,q_1,\dots, q_k}^{(3)})&=\frac{1}{q_1!\cdots q_k!}\, \frac{1}{k!}\ \sum_{j=0}^N {N\choose j}  (-1)^{N-j} j\, e^{-a_j}
\\
&=\frac{1}{q_1!\cdots q_k!}\, \frac{1}{k!}\, e^{-N-(k+1)}\ \sum_{j=0}^N {N\choose j}  (-1)^{N-j}\, j\, e^{j}
\\
&=\frac{1}{q_1!\cdots q_k!}\, \frac{1}{k!}\, e^{-(k+1)}\, N\, \Big(1-\frac{1}{e}\Big)^{N-1},
\end{align*}
where in the last step we have used the binomial theorem.

Now, summing in all tuples $q_1,\dots, q_k\ge 1$ with sum $N$, we get
\begin{align}\nonumber
\lim_{n\to\infty} \P_n(\T_{n;k,N}^{(3)})&=\frac{1}{k!}\, e^{-(k+1)}\, N\, \Big(1-\frac{1}{e}\Big)^{N-1}\ \sum_{\substack{q_1,\dots, q_k\ge 1\\q_1+\cdots+q_k=N}}
\frac{1}{q_1!\cdots q_k!}
\\
&=e^{-(k+1)}\, S(N,k)\, \frac{1}{(N-1)!}\, \Big(1-\frac{1}{e}\Big)^{N-1},\label{eq:prob of AkN3}
\end{align}
using \eqref{eq:product of reciprocal of factorials}.

Finally, summing over $k$ and $N$, we get
\begin{align}\nonumber
\lim_{n\to\infty} \P_n(\T_n^{(3)})&\stackrel{(\star)}{=}\sum_{k,N}e^{-(k+1)}\, S(N,k)\, \frac{1}{(N-1)!}\, \Big(1-\frac{1}{e}\Big)^{N-1}
\\ \label{eq:prob of A3}
&=\frac{1}{e}\ \sum_{k,N}S(N,k)\, \frac{(1-1/e)^{N-1}}{(N-1)!}\, \Big(\frac{1}{e}\Big)^k =
\frac{1}{e}\, \frac{1}{e}\,e^{1-1/e} \,e^{(e^{1-1/e}-1)/e}
\end{align}
(for the last identity, use
\eqref{eq:double-sum of S(n,k)1dev} with $x=1-1/e$ and $y=1/e$).

To justify the interchange of limit and (the double) sum in $(\star)$, we observe, from~\eqref{eq:Pn(Akq1...qk3)} and~\eqref{eq:cuenta para Pn(Akq1...qk3)}, that
$$
|\P_n(\T_{n;k,q_1,\dots, q_k}^{(3)})|\le \frac{1}{q_1!\cdots q_k!}\, \frac{1}{k!}\ \sum_{j=0}^N {N\choose j}\, j=\frac{1}{q_1!\cdots q_k!}\, \frac{1}{k!}\ N\,2^{N-1},
$$
and so
\begin{align*}
|\mathbf{P}_n(\T_{n;k,N}^{(3)})|&\le \frac{N}{k!}\, 2^{N-1}\sum_{\substack{q_1,\dots, q_k\ge 1\\q_1+\cdots+q_k=N}}\frac{1}{q_1!\cdots q_k!}
\\
&\le \frac{N}{k!}\ 2^{N-1}\, \textsc{coef}_N[e^{kz}]=\frac{1}{2}\, \frac{1}{k!}\, \frac{(2k)^N}{(N-1)!}.
\end{align*}
As
$$
\sum_{k\ge 1}\sum_{N\ge k} \frac{1}{k!}\, \frac{(2k)^N}{(N-1)!}<+\infty,
$$
dominated convergence justifies $(\star)$.
\end{proof}

\begin{remark}[Rooted labeled trees with the root farther away from the leaves]
{\upshape For $k\ge 0$, denote by $\T^{(k)}$ the subclass of rooted trees in which the root is, at least, $k$ units away from the boundary ($\partial_{\rm root}\ge k$). Write $\Psi_k(z)$ for its egf.

%Given a rooted tree $G$ on $n$ nodes $\{1,\dots,n\}$, call $\partial_{\text{root}}$ the minimum distance from the root to any leaf.

%For each $k\ge 0$, we aim to count (and determine the asymptotic behaviour as $n\to\infty$ of)
%the number of rooted labeled trees $G$ on $n$ nodes in which $\partial_{\text{root}}\ge k$.

For $k=0$, $\T^{(0)}=\T$, and the corresponding egf is just the Cayley's function, $\Psi_0(z)=T(z)$.

The symbolic method (see \cite{FS}) gives that the sequence $(\Psi_k(z))$ of egfs satisfies the recurrence relation
\begin{equation}\label{eq:recurrence of Psi}
\Psi_0(z)=T(z)\,,\quad \Psi_k(z)=z\,\big(e^{\Psi_{k-1}(z)}-1\big),\quad k\ge1.
\end{equation}
To see this, take a tree in $\T^{(k)}$, delete its root (and the edges departing from it), and observe that we are left with a non-empty set of rooted trees in~$\T^{(k-1)}$.

In particular, using Cayley's formula \eqref{eq:Cayley formula}, we get
\begin{equation}\label{eq:Psi1}
\Psi_1(z)=z\, \big(e^{\Psi_{0}(z)}-1\big)=z\big(e^{T(z)}-1\big)=T(z)-z.
\end{equation}
and
\begin{equation}\label{eq:Psi2}
\Psi_2(z)=z\, \big(e^{\Psi_{1}(z)}-1\big)=z\big(e^{T(z)-z}-1\big)=T(z)\, e^{-z}-z.
\end{equation}
In the latter case, the particular structure of $\Psi_2(z)$ allows to obtain the asymptotic behaviour of its coefficients in a direct manner (avoiding a combinatorial argument similar to that used in the proof of Theorem~\ref{thm:prob of delta3}), using a trick of Schur and Sz\'asz  (see \cite{FS},  Theorem VI.12, p.\ 434). The result in this case is that
$$
\lim_{n\to\infty} \mathbf{P}_n(\T_n^{(2)})=e^{-1/e}\approx 0\mbox{.}6922.
$$
For $k=3$, instead,
\begin{equation}\label{eq:Psi3}
\Psi_3(z)=z\, \big(e^{\Psi_{2}(z)}-1\big)=z\, e^{T(z) e^{-z}}\, e^{-z}-z,
\end{equation}
and the simple approach sketched above for $k=2$ does not work. That is why we had to go through the combinatorial argument of the proof of Theorem~\ref{thm:prob of delta3}, to obtain
$$
\lim_{n\to\infty} \mathbf{P}_n(\T_n^{(3)})=\frac{1}{e}\, e^{-1/e} \,e^{(e^{1-1/e}-1)/e}\approx 0\mbox{.}3522.
$$}
\end{remark}

Notice that the {\upshape height} of a rooted tree is the {\em maximum} distance from the root to the leaves, while the distance $\partial_{\rm  root}$ is the {\em minimum} distance from the root to the leaves. The egfs $\Phi_k(z)$ of trees of height $\le k$ satisfy
\begin{equation}\label{eq:recurrence of Phi}
\Phi_0(z)=z\,,\quad \Phi_k(z)=z\,e^{\Phi_{k-1}(z)},\quad k\ge1.
\end{equation}
The asymptotics of the proportion that rooted trees of height $k$ occupy in~$\T_n$ is well known, starting with the R\'enyi--Szekeres analysis of \eqref{eq:recurrence of Phi} (see~\cite{RS}).

It would be nice to have a general analogous analysis of the recurrence \eqref{eq:recurrence of Psi} that could lead to an answer for:
\begin{question}
For $k\ge 4$, and as $n \to \infty$, what is the proportion that trees in $\T_n^{(k)}$ do occupy in $\T_n$?
\end{question}

\subsubsection{Mean of the sum of degrees of descendants of the root}\label{section:mean of the sum}
In the same probability space $\T_n$ (rooted trees labeled with $\{1, \dots, n\}$, with uniform distribution), consider the random variable
$$
Y_n(T)=\unogrande_{\{\partial_{\text{root}}\ge 3\}}\cdot N_n(T)=\left\{\begin{array}{cl}
N_n(T)&\text{if $\partial_{\text{root}}\ge 3$,}
\\
0&\text{otherwise,}
\end{array}\right.
$$
where $N_n(T)$ is the number of nodes in the second generation of the graph~$T$ counted from the root (see Figure \ref{fig:tree3}). Observe that
$$
N_n(T)=\sum_{v \in  \lk(\text{\rm root of $T$})} ({\rm deg}(v)-1).
$$

The following asymptotic result holds.
\begin{thm}\label{thm:mean of N}
$$
\lim_{n\to\infty} \E_n(Y_n)=\Big[2-\frac{1}{e}+\frac{1}{e}\, \Big(1-\frac{1}{e}\Big)\,e^{1-1/e}\Big]=: d_3.
$$
\end{thm}

The numerical value of $d_3$ is $\approx 2\mbox{.}070$.
\begin{proof}
Recalling \eqref{eq:prob of AkN3} and \eqref{eq:prob of A3}, we observe that
\begin{align*}
\lim_{n\to\infty} \E_n(Y_n)&=\frac{1}{c_3}\, \sum_{k,N} N\cdot     S(N,k)\, \frac{(1-{1}/{e})^{N-1}}{(N-1)!}\, e^{-(k+1)}\,
\\
&=\frac{1}{c_3}\, \frac{1}{e} \, \sum_{k,N} N\cdot     S(N,k)\, \frac{(1-1/e)^{N-1}}{(N-1)!}\, \Big(\frac{1}{e}\Big)^k
\\
&=\frac{1}{c_3}\, \frac{1}{e}\,
\frac{1}{e}\, e^{1-1/e} \,e^{(e^{1-1/e}-1)/e}\, \Big[2-\frac{1}{e}+\frac{1}{e}\, \Big(1-\frac{1}{e}\Big)\,e^{1-1/e}\Big]
\\
&=\Big[2-\frac{1}{e}+\frac{1}{e}\, \Big(1-\frac{1}{e}\Big)\,e^{1-1/e}\Big]
\end{align*}
where we have used \eqref{eq:double-sum of S(n,k)2dev} with $x=1-1/e$ and $y=1/e$, and the value of $c_3$ from Theorem \ref{thm:prob of delta3}. The interchange of limit and double sum can be justified by dominated convergence, along the same lines as in the proof of Theorem~\ref{thm:prob of delta3}.\end{proof}

\subsubsection{From rooted to unrooted trees}\label{section:from rooted to unrooted}
Theorems \ref{thm:prob of delta3} and \ref{thm:mean of N} can be readily reinterpreted in the context of unrooted trees.

Fix $n$  and consider the collection $\mathcal{U}_n$ of the $n^{n-2}$ trees labeled with $\{1,  \dots, n\}$ endowed with the uniform probability.
For the sake of clarity, we denote probability and expectation in $\mathcal{U}_n$ with $\P_n^{\prime}$ and $\E_n^{\prime}$, respectively.

Let $X_n(T)$ denote the random variable in $\mathcal{U}_n$ that counts the number of deep nodes of~$T$:
$$
X_n(T)=\sum_{v\in V(T)} \unogrande_{\{\partial_{v}\ge 3\}}=|D(T)|
$$

Consider now a 0-1 matrix $M$, of dimensions $n\times n^{n-2}$, with columns labeled with $T_1, T_2,\dots$, the collection of trees in $\mathcal{U}_n$, and with rows labeled with the nodes $\{1,\dots, n\}$, where in the $(j,T_i)$-entry we place a 1 if the node~$j$ of~$T_i$ is deep, and we place a 0 otherwise.

Summing the entries of the matrix $M$ and dividing by $n^{n-2}$, we obtain the mean value of~$X_n$:
$$
\E^{\prime}_n(X_n)=\frac{1}{n^{n-2}}\sum_{T \in \mathcal{U}_n} X_n(T).
$$

Each (unrooted) tree $T_i$ leads to $n$ different rooted trees $T_i^{(1)}, \dots, T_i^{(n)}$ by choosing any of its nodes as the root; here, $T_i^{(j)}$ means that node $j$ has been selected as the root in the tree $T_i$.

Now build  a 0-1 matrix $M'$ of dimensions $n\times n^{n-1}$: rows are labeled with the $n$ nodes, and the columns with the collection of rooted trees in the following order: first $T_1^{(1)}, \dots, T_1^{(n)}$, then $T_2^{(1)}, \dots, T_2^{(n)}$, etc. The value of the entry $(v_i, T_j^{(k)})$ is 1 if $i=k$ and the node $i$ (the root of the tree $T_j^{(i)}$) is at distance $\ge 3$ to the boundary; and it is 0 otherwise.

The sum of the entries of $M'$, divided by $n^{n-1}$, gives the probability that in a rooted labeled tree, the root is at distance $\ge 3$ to its boundary. As the sum of the entries of $M'$ equals the sum of the entries of $M$, recalling Theorem \ref{thm:prob of delta3}, we deduce the following.
\begin{thm}\label{thm:expected number of deep nodes in unrooted trees}
As $n\to\infty$, the expectation of the proportion of nodes in a labeled tree on $n$ nodes that are at distance $\ge 3$ from any leaf tends to $c_3$, i.e.,
$$
\lim_{n\to\infty}\ \frac{1}{n}\,\E'_n(X)=c_3.
$$
\end{thm}

Next, in the probability space $\mathcal{U}_n$ of unrooted trees $T$ labeled with $\{1, \dots, n\}$ and endowed with uniform probability, consider the random variable
$$
\Upsilon(T)=\sum_{v\in V(T)} \unogrande_{\{\partial_{v}\ge 3\}}\cdot N_v=\sum_{v \in \D(T)} N_v
$$
where $N_v$ is the number of nodes two units away from $v$. Observe that
$$
N_v(T)=\sum_{w \in \lk(v)} ({\rm deg}(w)-1).
$$

An analogous argument as above using Theorem~\ref{thm:mean of N} instead of Theorem~\ref{thm:prob of delta3}  yields:
\begin{thm}\label{thm:star in unrooted trees}
$$
\lim_{n\to\infty}\ \frac{1}{n}\,\E'_n(\Upsilon)=d_3.
$$
\end{thm}

%Proposition \ref{thm:average value of upsilon} is just a restatement of Theorem \ref{thm:star in unrooted trees}.

\end{document}